\documentclass[reqno,12pt]{amsart}
\usepackage{txfonts}
\usepackage{mathrsfs}
\usepackage{amsmath}
\usepackage{amssymb}

\usepackage{amsthm}
\usepackage{verbatim}
\usepackage{latexsym, bm}
\usepackage{diagrams}

\title[Inequalities of Chern classes ]{Inequalities of Chern classes on nonsingular projective $n$-folds of Fano and general type with ample canonical bundle}

\author[Rong Du]{Rong Du$^{\dag}$}
\address{Department of Mathematics\\
Shanghai Key Laboratory of PMMP\\
East China Normal University\\
Rm. 312, Math. Bldg, No. 500, Dongchuan Road\\
Shanghai, 200241, P. R. China} \email{rdu@math.ecnu.edu.cn}

\author[Hao Sun]{Hao Sun$^{\dag\dag}$}

\address{Department of Mathematics, Shanghai Normal University,
Shanghai 200234, P. R. of China}

\email{hsun@shnu.edu.cn}

\thanks{$^{\dag}$ The Research is Sponsored by the National Natural Science Foundation of China (Grant No. 11471116, 11531007) and Science and Technology Commission of Shanghai Municipality (Grant No. 18dz2271000).}
\thanks{$^{\dag\dag}$ The Research is Sponsored by the National Natural Science Foundation of China (Grant No. 11771294, 11301201)}

\theoremstyle{definition}
\newtheorem{theorem}[subsection]{Theorem}
\newtheorem{lemma}[subsection]{Lemma}

\newtheorem{proposition}[subsection]{Proposition}
\newtheorem{example}{Example}[section]
\newtheorem{corollary}[subsection]{Corollary}
\newtheorem{remark}[subsection]{Remark}

\newfont{\drnew}{wncyr10}

\allowdisplaybreaks

\def\dashfill{\leaders\hbox{\hbox to 3.25pt{\hrulefill}\hspace*{2pt}\hbox to 3.25pt{\hrulefill}}\hfill}
\newcommand{\CITE}[1]{{[#1]}}
\let\cite=\CITE

\begin{document}

\begin{abstract}
Let $X$ be a nonsingular projective $n$-fold $(n\ge 2)$ of Fano or of general type with ample canonical bundle $K_X$ over an algebraic closed field $\kappa$ of any characteristic. We produce a new method to give a bunch of inequalities in terms of all the Chern classes $c_1, c_2, \cdots, c_n$ by pulling back Schubert classes in the Chow group of Grassmannian under the Gauss map. Moreover, we show that if the characteristic of $\kappa$ is $0$, then the Chern ratios $(\frac{c_{2,1^{n-2}}}{c_{1^n}}, \frac{c_{2,2,1^{n-4}}}{c_{1^n}}, \cdots, \frac{c_{n}}{c_{1^n}})$ are contained in a convex polyhedron for all $X$. So we give an affirmative answer to a generalized open question, that whether the region described by the Chern ratios is bounded, posted by Hunt (\cite{Hun}) to all dimensions. As a corollary, we can get that there exist constants $d_1$, $d_2$, $d_3$ and $d_4$ depending only on $n$ such that $d_1K_X^n\le\chi_{top}(X)\le d_2 K_X^n$ and $d_3K_X^n\le\chi(X, \mathscr{O}_X)\le d_4 K_X^n$. If the characteristic of $\kappa$ is positive, $K_X$ (or $-K_X$) is ample and $\mathscr{O}_X(K_X)$ ($\mathscr{O}_X(-K_X)$, respectively) is globally generated, then the same results hold.
\end{abstract}

\maketitle

\vspace{1cm}
\section{\textbf{Introduction}}
One of the fundamental questions in algebraic geometry is the
classification of algebraic varieties. The classical method is by
considering the numerical invariants first. Among all kinds of
numerical invariants, Chern numbers are the most natural and
important ones. The terminology geography which was first introduced
by Persson in 1981 (see \cite{Per}) is used to describe the
distribution of Chern numbers of nonsingular projective varieties of
general type, i.e. whether there exists nonsingular projective
varieties $X$ of dimension $n$ for every given set of numbers such
that $X$ has precisely those Chern numbers. So the first thing is to
determine if the quotients of those numbers are in a bounded set or
not. For $n=2$ and $X$ is minimal over the complex numbers, we have
Noether's inequalities:

\[K_X^2\ge 2p_g-4,\] here $p_g=h^0(X, K_X)$.
From Noether's formula, we can have
\[5c_1^2+36\ge c_2.\]
On the other hand, we have the famous Bogomolov-Miyaoka-Yau
inequality
\[c_2\ge \frac{1}{3}c_1^2.\]
So $c_2/c_1^2$ is bounded.

For $n=2$ and over a field of positive characteristic, Noether's
inequality (see \cite{Lie}) and Noether's formula (see \cite{Bad}
Chap. 5) remain true, while Lang (\cite{Lan}) and Easton (\cite{Eas}) gave examples of surfaces, such as generalized Raynaud surfaces, for which the Bogomolov-Miyaoka-Yau inequality no longer holds (cf. \cite{Szp}, Sec. 3.4). In fact, from
Raynaud's example (\cite{Ray}) even the weaker inequality $c_2\ge 0$
due to Castelnuovo and de Franchis fails (cf. \cite{Gu}, Sec. 3). So
it is natural to formulate an inequality in positive characteristic
bounding $c_2$ from below by $c_1^2$. Shepherd-Barron has already
considered a similar question and proved that $\chi> 0$
(equivalently, $c_2>-c_1^2$) with a few possible exceptional cases
when the characteristic of the field is not greater than $7$
(\cite{S-B}, Theorem 8). Later, Gu solved Shepherd-Barron's question
and got an optimal lower bound of $\chi/c_1^2$.

For $n=3$ and over the complex numbers, Hunt initiated the study of
geography for threefolds (\cite{Hun}). Later, Chang, Kim and Nollet
gave a bound of $c_3$ by quadratic forms in term of $c_1c_2$ and
$c_1^3$ of threefolds with ample canonical bundle (\cite{C-K-N}).
Later, Chang and Lopez obtained a linear bound of $c_3$ of threefolds
with ample canonical bundle, i.e. boundedness for the region
described by the Chern ratios $c_3/c_1c_2,~c_1^3/c_1c_2$
(\cite{C-L}). Their idea is to control the Euler number by the
Rieman-Roch formula and cohomology calculations. Their result relies
on the Bogomolov-Miyaoka-Yau inequality heavily for dimension $3$.
M. Chen-Hacon and J. Chen- Hacon (see \cite{CM-H}, \cite{CJ-H}) also
considered the geography of Gorenstein minimal complex $3$-folds of
general type after 2000. For positive characteristic, as far as the
authors know that there are no such kind of considerations.

For $n\ge 4$, such problem seems unknown even over the field of
complex numbers. Except for the famous Bogomolov-Miyaoka-Yau
inequality which is in the case of characteristic $0$,  many other
mathematicians studied inequalities of Chern classes (see
\cite{F-L}, \cite{B-S-S}, \cite{Ca-Sc}, \cite{Ch-Le}, \cite{Ko}, \cite{Mi},
\cite{Sun}). In this paper, we show that there are similar boundedness
result as the dimension $3$ for characteristic $0$ if $K_X$ (or
$-K_X$) is ample and for positive characteristic if moreover
$\mathscr{O}_X(K_X)$ ($\mathscr{O}_X(-K_X)$, respectively) is globally generated.
So we give an affirmative answer to a generalized open question,fore
that whether the region described by the Chern ratios is bounded,
posted by Hunt (\cite{Hun}) to all dimensions. In particular, we generalize Chang-Lopez's result to all dimensions. (For the notations, please see the paragraph before Theorem \ref{mainT}.)

\vspace{.5cm}
\textbf{Main Theorem:}(see Theorem \ref{mainT})

Let $X$ be a nonsingular projective variety of dimension $n$ over an
algebraic closed field $\kappa$ with any characteristic. Suppose
$K_X$ (or $-K_X$) is ample. If the characteristic of $\kappa$ is $0$
or the characteristic of $\kappa$ is positive and $\mathscr{O}_X(K_X)$ ($\mathscr{O}_X(-K_X)$, respectively )
is globally generated, then
\begin{equation}\label{poly}
(\frac{c_{2,1^{n-2}}}{c_{1^n}}, \frac{c_{2,2,1^{n-4}}}{c_{1^n}}, \cdots, \frac{c_{n}}{c_{1^n}})\in \mathbb{A}^{p(n)}
\end{equation}
is contained in a convex polyhedron in $\mathbb{A}^{p(n)}$
independent of $X$, where $p(n)$ is the partition number and the
elements in the parentheses arranged from small to big in terms of
the alphabet order of the lower indices of the numerators.

\vspace{.5cm} In particular, we show that the Euler number
$\chi_{top}(X)$ and Euler characteristic of the structure sheaf $\chi(X, \mathscr{O}_X)$ can be controlled linearly by $K_X^n$. More
precisely, there exist constants $d_1$, $d_2$, $d_3$ and $d_4$, which do not depend on $X$ but only on $n$, such that
\begin{equation*}
d_1K_X^n\le\chi_{top}(X)\le d_2 K_X^n
\end{equation*}
and
\begin{equation*}
d_3K_X^n\le\chi(X, \mathscr{O}_X)\le d_4 K_X^n.
\end{equation*}
Our results can also infer the classical boundedness result for
dimension $2$ and Chang-Lopez's result for dimension $3$ easily
without using the Bogomolov-Miyaoka-Yau inequality. Furthermore, we
can deduce a bunch of inequalities of Chern classes not only Chern
numbers.

\begin{remark}
Fulton-Lazarsfeld,  Demailly-Peternell-Schneider, and Catanese-Schneider's results can give many inequalities of Chern
classes. Moreover, combining their results and
Bogomolov-Miyaoka-Yau inequalities (for characteristic $0$), one
gets immediately bounds for $\frac{c_{2,1^{n-2}}}{c_{1^n}}$,  which
can also start our induction in the proof of the main thoerem (see
Theorem \ref{mainT}). However, their methods depend on the characteristic $0$ and Bogomolov-Miyaoka-Yau inequality.  Our method is totally different and doesn't rely on Bogomolov-Miyaoka-Yau inequality. We can deal with any characteristic of the algebraic closed field simultaneously.
\end{remark}

In Section 2, we introduce Schubert cycles and Schubert classes of
Grassmannian. Pieri's formula and Giambelli's formula are also
mentioned in this section for Schubert calculus later. In Section
3, we recall Fujita conjecture and known relevant very ampleness
results in any characteristic first. Then we solve our main theorem
by estimating Chern classes from Schubert calculus and the help of
Zak's theorem. In Section 4, we produce a new algorithm to give a
bunch of inequalities in terms of all the Chern classes $c_1, c_2,
\cdots, c_n$ by pulling back the Schubert classes in the Chow group of
the Grassmannian under the Gauss map.

\section{\textbf{Schubert cycles and Schubert classes of Grassmannian}}
We will recall the basic definition Schubert cycles and classes of
the Chow group of $G(n, k)$, the Grassmannian of k-dimensional
subspaces in an n-dimensional vector space $V$, and analyze their
intersections, a subject that goes by the name of the \emph{Schubert
calculus}. Of course we may also consider $G(n, k)$ in its
projective guise as $\mathbb{G}(n-1, k-1)$, the Grassmannian of
projective $(k-1)$-planes in $\mathbb{P}^{n-1}$. We recommend
excellent books \cite{G-H} and \cite{E-H} for details.

\emph{Schubert cycles} are defined in terms of a chosen complete flag $\mathcal{V}$ in $V$ , i.e., a nested
sequence of subspaces
\[0\subset V_1 \subset \cdots \subset V_{n-1}\subset V_n=V\]
with dim$V_i=i$. The Schubert cycles are indexed by sequences
$\overrightarrow{a}=(a_1, a_2, \cdots a_k)$ of integers with
\[n-k\ge a_1\ge a_2\ge \cdots\ge a_k\ge 0\] We define
$|\overrightarrow{a}|:=\sum_{i=1}^k a_i$ and
$l(\overrightarrow{a}):=k$.

For such a sequence $\overrightarrow{a}$, we define the
\emph{Schubert cycle}
$\Sigma_{\overrightarrow{a}}(\mathcal{V})\subset G(n, k)$ to be the
closed subset
\[\Sigma_{\overrightarrow{a}}(\mathcal{V})=\{\Lambda\in G(n, k)~|~\text{dim}(V_{n-k+i-a_i}\cap\Lambda)\ge i~ \text{for all}~i\}.\]
We know that the class
$[\Sigma_{\overrightarrow{a}}(\mathcal{V})]\in A(G(n, k))$ does not
depend on the choice of the flag, since any two flags differ by an
element of GL$_n$, where $A(G(n, k))$ is the Chow group of $G(n,
k)$. So we shorten the notation to $\Sigma_{\overrightarrow{a}}$ and
define \emph{Schubert classes}
\[\sigma_{\overrightarrow{a}}:=[\Sigma_{\overrightarrow{a}}]\in A(G(n, k)).\]
The following theorem shows that $A(G(n, k))$ is a free abelian
group and that the classes $\sigma_{\overrightarrow{a}}$ form a
basis.

\begin{theorem}(\cite{E-H} Corollary 4.7)
The Schubert classes form a free basis for $A(G(n, k))$, and the intersection
form $$A^m(G(n, k))\times A^{\text{dim}G(n, k)-m}(G(n, k))\rightarrow\mathbb{Z}$$ have the Schubert classes as dual bases.
\end{theorem}

To simplify notation, we generally suppress trailing zeros in the
indices, writing $\sigma_{a_1,\cdots,~ a_s}$ in place of
$\sigma_{(a_1,\cdots, ~a_s, 0, \cdots,0)}$. Also, we use the
shorthand $\sigma_{p^r}$ to denote $\sigma_{p,\cdots, p}$ with $r$
indices equal to $p$.

Let $\mathscr{V}:=G(n, k)\times V$ be the trivial vector bundle of rank $n$ on $G(n, k)$ whose fiber
at every point is the vector space $V$. We write $S$ for the rank-$k$ subbundle of $\mathscr{V}$ whose
fiber at a point $\Lambda\in G(n, k)$ is the subspace $\Lambda$ itself; that is,
\[S_{[\Lambda]}=\Lambda\subset V=\mathscr{V}_{[\Lambda]}.\]

$S$ is called the \emph{universal subbundle} on $G(n, k)$; the quotient $Q=\mathscr{V}/S$  is called the \emph{universal quotient bundle},i.e.,
\begin{equation}\label{tau}
0\rightarrow S\rightarrow \kappa^{n}\rightarrow Q\rightarrow 0.
\end{equation}

\begin{proposition}(see \cite{E-H} Sec. 5.6.2 Grassmannians or \cite{G-H} Sec. 3.3)\label{cS}
\[c_pS=(-1)^p\sigma_{1^p}.\]
\end{proposition}

Next we will talk about the intersection of Schubert classes.
One situation in which we can give a simple formula for the product of Schubert
classes is when one of the classes has the special form $\sigma_b$ with integer $b$. Such classes are
called \emph{special Schubert classes}.
\begin{proposition}(Pieri's Formula)
For any Schubert class $\sigma_{\overrightarrow{a}}$ and any integer $b$,
\[\sigma_b\sigma_{\overrightarrow{a}}=\sum_{|\overrightarrow{c}|=|\overrightarrow{a}|+b\atop a_i\le c_i\le a_{i-1},~ \forall i}\sigma_{\overrightarrow{c}}.\]
\end{proposition}

Pieri's formula tells us how to intersect an arbitrary Schubert class with one of
the special Schubert classes $\sigma_b$ with integer $b$. Giambelli's formula is complementary, in
that it tells us how to express an arbitrary Schubert class in terms of special ones; the
two together give us a way of calculating the product of two arbitrary
Schubert classes.

\begin{proposition}(Giambelli's Formula)
\begin{equation}
\sigma_{a_1,a_2,\cdots, a_q}=\left|
                               \begin{array}{ccccc}
                                 \sigma_{a_1} & \sigma_{a_1+1} & \sigma_{a_1+2} & \cdots & \sigma_{a_1+q-1} \\
                                 \sigma_{a_2-1} & \sigma_{a_2} & \sigma_{a_2+1} & \cdots & \sigma_{a_2+q-2} \\
                                 \sigma_{a_3-2} & \sigma_{a_3-1} & \sigma_{a_3} & \cdots & \sigma_{a_3+q-3} \\
                                 \vdots & \vdots & \vdots & \ddots & \vdots \\
                                 \sigma_{a_q-q+1} & \sigma_{a_q-q+2} & \sigma_{a_q-q+3} & \cdots & \sigma_{a_q} \\
                               \end{array}
                             \right|
\end{equation}
\end{proposition}

\section{\textbf{A linear bound on the Chern ratios}}
Let $X$ be a nonsingular projective variety of dimension $n$ $(n\ge
2)$ over an algebraic closed field $\kappa$ with any characteristic.
Suppose $K_X$ or $-K_X$ is ample and $mK_X$ is very ample ($m$ can
be negative if $-K_X$ is ample, i.e. $X$ is Fano).

When $X$ is a complex surface and $L$ is an ample line bundle on
$X$, Reider (\cite{Rei}) showed that $K_X+3L$ is always generated by
global sections and $K_X+4L$ very ample. Around the same period,
Fujita (\cite{Fuj}) raised the following interesting conjecture.

\textbf{Fujita's Conjecture: } Let $X$ be a smooth n-dimensional
complex projective algebraic variety and let $L$ be an ample divisor
on X.
\begin{enumerate}
\item For $t\ge n+1$, $tL+K_X$ is base point free.
\item For $t\ge n+2$, $tL+K_X$ is very ample.
\end{enumerate}

For the very ampleness conjecture, one of the first results proved
in dimension $n\ge 3$ is the very ampleness of $2K_X + 12n^nL$,
using an analytic method based on the solution of a
Monge-Amp$\grave{\text{e}}$re equation by Demailly (see \cite{Dem1}). Other
related works are (\cite{Dem2}, \cite{E-L-N}, \cite{Siu1},
\cite{Siu2}, \cite{Siu3}, \cite{Yeu}) to improve the effective bound.

Since those proofs rely on the Kodaira Vanishing Theorem and its
generalizations, it seems to us that we don't have such results for
positive characteristic. However, Smith proved another version of
the Fujita conjecture in arbitrary characteristic if $L$ is ample
and generated by global sections via tight closure theory (see
\cite{Smi1}, \cite{Smi2}). Later, Keeler used the method of positive
characteristic to show another version of Fujita's Conjecture (see
\cite{Kee}).

\begin{theorem} (see \cite{Kee}, Theorem 1.1) \label{Keeler}
Let $X$ be a projective scheme of pure dimension $n$, smooth over a
field $\kappa$ of arbitrary characteristic. Let $L$ be an ample,
globally generated line bundle and let $H$ be an ample line bundle.
Then
\begin{enumerate}
\item $K_X+nL+H$ is base point free.
\item $K_X+(n+1)L+H$ is very ample.
\end{enumerate}
\end{theorem}

We will use the results of Fujita's very ampleness conjecture to get
a bound which depends only on the dimension $n$ for the Chern ratios
and will use Van de Ven's idea first (cf. \cite{Hun} Introduction).

Assume $i: X\hookrightarrow \mathbb{P^N}$ is the canonical embedding
defined by the linear system $|mK_X|$ (i.e.
$mK_X=i^*\mathscr{O}_{\mathbb{P}^N}(1)$).  Let $\gamma$ be the Gauss
map:

\begin{equation} \label{Gauss}
\begin{split}
\gamma:~&~ X \longrightarrow G(\mathbb{P}^N, \mathbb{P}^n)=G(N+1, n+1)\\
        &~ x \longmapsto     T_{X,x},
\end{split}
\end{equation}
where $T_{X,x}$ is the tangent space to $X$ at $x$. There is an
usual bundle sequence on $G(N+1, n+1)$:
\[0\rightarrow S\rightarrow \kappa^{N+1}\rightarrow Q\rightarrow 0,\]
where $S$ is the universal bundle (see \cite{G-H} Chapter I or
\cite{E-H} Chapter 3), which pulls back to an exact sequence on $X$
\begin{equation}\label{pulltau}
0\rightarrow \gamma^*S\rightarrow \gamma^*\kappa^{N+1}\rightarrow \gamma^*Q\rightarrow 0.
\end{equation}

On the other hand we have the Euler exact sequence on
$\mathbb{P}^N$:
\[0\rightarrow \mathscr{O}_{\mathbb{P}^N}(-1)\rightarrow \mathscr{O}_{\mathbb{P}^{N}}^{N+1}\rightarrow T_{\mathbb{P}^N}(-1)\rightarrow 0,\]
which pulls back to an exact sequence on $X$
\begin{equation}\label{Euler}
0\rightarrow i^*\mathscr{O}_{\mathbb{P}^N}(-1)\rightarrow i^*\mathscr{O}_{\mathbb{P}^{N}}^{N+1}\rightarrow i^*T_{\mathbb{P}^N}(-1)\rightarrow 0.
\end{equation}
Moreover we have the twisted adjunction sequence on $X$:
\begin{equation}\label{ad}
0\rightarrow T_X(-1)\rightarrow i^*T_{\mathbb{P}^N}(-1)\rightarrow N_{X/\mathbb{P}^N}(-1)\rightarrow 0.
\end{equation}
These three sequences (\ref{pulltau}), (\ref{Euler}) and (\ref{ad}) fit together in a diagram
\begin{diagram}[small]
 &     &                                &      & 0                                     &            &0                   &&\\
 &     &                                &      &\dTo                                   &            &\dTo                  &&\\
 &     &                                &      &\gamma^*S                              &\rTo        &T_X(-1)                &&\\
 &     &                                &      &\dTo                                   &            &\dTo                     &&\\
0&\rTo &i^*\mathscr{O}_{\mathbb{P}^N}(-1)&\rTo  &i^*\mathscr{O}_{\mathbb{P}^N}^{N+1}     &\rTo      &i^*T_{\mathbb{P}^N}(-1) &\rTo  &0\\
 &     &                                &      &\dTo                                   &            &\dTo                     &&\\
  &     &                                &      &\gamma^*Q                             &       &N_{X/\mathbb{P}^N}(-1)               &&\\
   &     &                                &      &\dTo                                   &            &\dTo                     &&\\
    &     &                                &      & 0                                     &            &0                   &&.
\end{diagram}
By the snake lemma, we have the exact sequence
\[0\rightarrow i^*\mathscr{O}_{\mathbb{P}^N}(-1)\rightarrow \gamma^*S \rightarrow T_X(-1)\rightarrow 0,\]
i.e.
\[0\rightarrow \mathscr{O}_X(-mK_X)\rightarrow \gamma^*S \rightarrow T_X(-mK_X)\rightarrow 0.\]

So
\begin{equation}\label{cX}
c(\gamma^*S)=c(\mathscr{O}_X(-mK_X))c(T_X(-mK_X)).
\end{equation}

\begin{lemma}
Let $E$ be a vector bundle of rank $r$ and $L$ be a line bundle on a scheme $X$ over an algebraic closed field $k$. Then for all $p\ge 0$,
\begin{equation}
c_p(E\otimes L)=\sum_{i=0}^p{r-i\choose p-i}c_i(E)c_1(L)^{p-i}.
\end{equation}
\end{lemma}
From the above lemma, we have the following result.
\begin{lemma}\label{cp}
\begin{equation}
c(T_X(-mK_X)=\sum_{p=0}^{n}\sum_{i=0}^{p}{n-i\choose p-i} m^{p-i} c_ic_1^{p-i}.
\end{equation}
\end{lemma}

Next, we will prove our main theorem. The following two lemmas are needed in the proof.
\begin{lemma}\label{det}
Let $D_0=1$, and for any positive integer $n$ let
\begin{equation*}
D_n=\left|
                            \begin{array}{ccccc}
                                 a_{1} & a_{2} & a_{3} & \cdots & a_{n} \\
                                 1 & a_{1} & a_{2} & \cdots & a_{n-1} \\
                                 & 1 & a_{1} & \cdots & a_{n-2} \\
                                  &  & \ddots & \ddots & \vdots \\
                                  \multicolumn{2}{c}{\raisebox{1.3ex}[0pt]{\huge0}}
                                   &  & 1 & a_{1} \\
                               \end{array}
                             \right|,
\end{equation*}
then one has $$\sum_{i=0}^n(-1)^iD_ia_{n-i}=0,$$ here $a_0=1$.
\end{lemma}
\begin{proof}
Using the expansion $D_n$ along the first row, one sees that
\begin{eqnarray*}
D_n&=&\sum_{i=1}^n(-1)^{1+i}a_i\left| \begin{array}{cccccccc}
                                 1 & a_{1} & \cdots & a_{i-2} & a_{i}&a_{i+1}& \cdots & a_{n-1} \\
                                 &\ddots & \ddots&\vdots&\vdots&\vdots & &\vdots\\
                                 & & 1 & a_1& a_3&a_4&\cdots &a_{n-i+2}\\
                                 & & & 1 & a_2&a_3& \cdots & a_{n-i+1}\\
                                 & & & & a_1 & a_2&\cdots & a_{n-i}\\
                                 & & & & 1 &a_1& \cdots & a_{n-i-1}\\
                                 & & & &   &\ddots&\ddots&\vdots\\
                                 \multicolumn{2}{c}{\raisebox{1.3ex}[0pt]{\Huge0}}
                                 & & &   &      & 1   & a_1\\
                               \end{array}\right|\\
 &=&\sum_{i=1}^n(-1)^{1+i}a_iD_{n-i}.
\end{eqnarray*}
Hence $$\sum_{i=0}^n(-1)^iD_ia_{n-i}=0.$$
\end{proof}

We say cycles $\delta_1$ and $\delta_2$ satisfying $\delta_1\le\delta_2$ if and only if $\delta_2-\delta_1$ is a nonnegative cycle.
\begin{lemma}\label{contr}
$\sigma_{1^t}\le \sigma_1^t,$ where $t$ is a positive integer.
\end{lemma}
\begin{proof}
By Pieri's Formula and induction on $t$, we have
\[\sigma_{1^t}\le\sigma_{1^{t-1}}\sigma_{1}\le\sigma_1^t.\]
\end{proof}

The crucial idea for proving the main theorem is pulling back
effective Schubert classes of the Chow group of the Grassmannian
under the Gauss map after Schubert calculating in the Grassmannian.
So we need to guarantee the intersection theory of Schubert cycles
can be kept under the Guass map which is true by Zak's theorem.
\begin{theorem}\label{Zak}(\cite{Zak}, Corollary 2.8)
Let $X$ be a nonsingular projective variety of dimension $n$ and
$X\neq \mathbb{P}^n$ over an algebraically closed field $\kappa$. Then
the Gauss map is finite. If in addition the characteristic of
$\kappa$ is $0$, then the Gauss map is the normalization morphism.
\end{theorem}

Let $n$ be a positive integer. Denote $p(n)$ to be the partition number of $n$, i.e. the number of the way to express $n$ as the summation of positive integers without considering the orders of them. We can define alphabet order of all vectors as follows. Given any two vectors $\overrightarrow{a}=(a_1,a_2,\cdots,a_r)$ and $\overrightarrow{b}=(b_1,b_2,\cdots,b_s)$ such that $a_1\ge a_2\ge\cdots\ge a_r>0$ and $b_1\ge b_2\ge\cdots\ge b_s>0$. Suppose $r=s$ otherwise we just put $0$'s after the short one such that $l(\overrightarrow{a})=l(\overrightarrow{b})$. If $a_1>b_1$ then we denote $\overrightarrow{a}>\overrightarrow{b}$. Otherwise if $a_1=b_1$, then we compare $a_2$ and $b_2$. Without loss of generality, suppose $a_2>b_2$ then we denote $\overrightarrow{a}>\overrightarrow{b}$.  Otherwise if $a_1=b_1$ and $a_2=b_2$, then we compare $a_3$ and $b_3$ and keep going. So we can compare the order of any two vectors $\overrightarrow{a}$ and  $\overrightarrow{b}$.

Suppose $a_1\ge a_2\ge\cdots\ge a_r>0$.
Denote $c_{\overrightarrow{a}}=c_{a_1,a_2,\cdots,a_r}:=c_{a_1}c_{a_2}\cdots c_{a_r}$, $c_{\overrightarrow{a}}S=c_{a_1,a_2,\cdots,a_r}S:=c_{a_1}S c_{a_2}S \cdots c_{a_r}S$ and $c_i^tS:=(c_iS)^t.$
\begin{theorem}\label{mainT}
Let $X$ be a nonsingular projective variety of dimension $n$ over an
algebraically closed field $\kappa$ with any characteristic. Suppose
$K_X$ (or $-K_X$) is ample. If the characteristic of $\kappa$ is $0$ or
the characteristic of $\kappa$ is positive and $\mathscr{O}_X(K_X)$ ($\mathscr{O}_X(-K_X)$ ,respectively)
is globally generated, then
\begin{equation}\label{poly}
(\frac{c_{2,1^{n-2}}}{c_{1^n}}, \frac{c_{2,2,1^{n-4}}}{c_{1^n}}, \cdots, \frac{c_{n}}{c_{1^n}})\in \mathbb{A}^{p(n)}
\end{equation}
is contained in a convex polyhedron in $\mathbb{A}^{p(n)}$ independent of $X$, where the elements in the parentheses are arranged from small to big in terms of the alphabet order of the lower indices of the numerators.
\end{theorem}
\begin{proof}
Assume $X\neq \mathbb{P}^n$ without loss of generality, because finite objects will not affect the result.
Suppose $mK_X$ is very ample ($m$ can be negative if $-K_X$ is ample, i.e. $X$ is Fano), where $m$ only depends on $n$ by the known results of Fujita's very ampleness conjecture.
By Proposition \ref{cS} and Lemma \ref{contr}, we have
\begin{equation}
0\le (-1)^nc_{\overrightarrow{a}}S=(-1)^nc_{a_1,a_2,\cdots,a_r}S\le (-1)^nc_{1^n}S,
 \end{equation}
 for any $\overrightarrow{a}$ with $a_1\ge a_2\ge\cdots\ge a_r>0$ and $|\overrightarrow{a}|=\sum_{i=1}^r a_i=n$. Let $\gamma$ be the Gauss map defined in (\ref{Gauss}). From Zak's Theorem \ref{Zak}, $\gamma$ is finite. Then
\begin{equation}\label{lessthan}
0\le (-1)^nc_{\overrightarrow{a}}(\gamma^*S)\le (-1)^n(c_{1^n}(\gamma^*S)).
\end{equation}
We only need to show that each element in the $q$-th coordinate is bounded independent of $X$. We use induction on $q$. For $q=1$,
\begin{equation}
0\le (-1)^nc_{2,1^{n-2}}(\gamma^*S)\le (-1)^n(c_{1^n}(\gamma^*S)).
\end{equation}
Form (\ref{cX}) and Lemma \ref{cp}, we have
\begin{eqnarray}
c_1(\gamma^*S)&= & c_1(T_X(-mK_X))+mc_1\\\nonumber
&= & (nm+1)c_1+mc_1\\\nonumber
&= & ((n+1)m+1)c_1.
\end{eqnarray}

\begin{eqnarray}
c_2(\gamma^*S)&= & c_2(T_X(-mK_X))+mc_1c_1(T_X(-mK_X))\\\nonumber
&= & (\frac{1}{2}n(n-1)m^2+(n-1)m)c_1^2+c_2+mc_1(nm+1)c_1\\\nonumber
&= & (\frac{1}{2}n(n+1)m^2+nm)c_1^2+c_2.
\end{eqnarray}

So
\[\begin{split}
0\le(-1)^n((\frac{1}{2}n(n+1)m^2+nm)c_1^2+c_2)((n+1)m+1)^{n-2}c_1^{n-2}\\
\le (-1)^n((n+1)m+1)^nc_1^n,
\end{split}\]
and $\frac{c_{2,1^{n-2}}}{c_{1^n}}$ is bounded independent of $X$, i.e., $q=1$ is correct.

Now for any $\overrightarrow{a}$ with $a_1\ge a_2\ge\cdots\ge a_r>0$
and $\sum_{i=1}^r a_i=n$, by Lemma \ref{cp} and (\ref{cX}), one sees
\begin{equation}
c_p(T_X(-mK_x))=\sum_{i=0}^{p}{n-i\choose p-i} m^{p-i} c_ic_1^{p-i}
\end{equation}
and
\begin{equation}
c_p(\gamma^*S)=c_p(T_X(-mK_X))+c_{p-1}(T_X(-mK_X))mc_1.
\end{equation}
So
\begin{equation}\label{cpgen}
c_p(\gamma^*S)=\sum_{i=0}^{p}{n-i\choose p-i} m^{p-i} c_ic_1^{p-i}+\sum_{i=0}^{p-1}{n-i\choose p-1-i} m^{p-i} c_ic_1^{p-i}.
\end{equation}
From (\ref{lessthan}), it follows
\[0\le (-1)^nc_{\overrightarrow{a}}(\gamma^*S)\le (-1)^n((n+1)m+1)^nc_1^n.\]
We can see that the biggest lower index of Chern class in the left hand side of (\ref{cpgen}) is $\overrightarrow{a}$,
so we can control the value of
$\frac{c_{\overrightarrow{a}}}{c_{1^n}}$
in terms of $\frac{c_{\overrightarrow{b}}}{c_{1^n}}$, for every $\overrightarrow{b}<\overrightarrow{a}$. By induction, each $\frac{c_{\overrightarrow{b}}}{c_{1^n}}$ is independent of $X$, so we are done.
\end{proof}

By the proof of the above theorem, we can have the following result easily.

\begin{corollary}
Let $X$ be a nonsingular projective variety of dimension $n$ over an
algebraically closed field $\kappa$ with any characteristic. Suppose
$K_X$ (or $-K_X$) is ample. If the characteristic of $\kappa$ is $0$
or the characteristic of $\kappa$ is positive and
$\mathscr{O}_X(K_X)$ ($\mathscr{O}_X(-K_X)$, respectively) is globally generated,
then the Euler number $\chi_{top}(X)$ and Euler characteristic of the structure sheaf $\chi(X, \mathscr{O}_X)$ quotient by $K_X^n$ is
bounded, i.e. there exist constants $d_1$, $d_2$, $d_3$ and $d_4$, which do not depend on $X$ but only on $n$, such that
\begin{equation}\label{chitop}
d_1K_X^n\le\chi_{top}(X)\le d_2 K_X^n
\end{equation}
and
\begin{equation}\label{chi}
d_3K_X^n\le\chi(X, \mathscr{O}_X)\le d_4 K_X^n.
\end{equation}
\end{corollary}
\begin{proof}
Since $\chi_{top}(X)=c_n$, (\ref{chitop}) holds by Theorem \ref{mainT}.

For (\ref{chi}), it is direct result from Theorem \ref{mainT} and Grothendieck-Hirzebruch-Riemann-Roch theorem since the Todd classs of $X$ can be express as the linear combination of Chern classes.
\end{proof}

\section{\textbf{equalities of Chern classes}}
Let $X$ be a nonsingular projective $n$-fold $(n\ge 2)$ of Fano or of general type with ample canonical bundle $K_X$ over an algebraically closed field $\kappa$ of any characteristic. We will produce a new method to give a bunch of inequalities in terms of all the Chern classes $c_1, c_2, \cdots, c_n$ by pulling back Schubert classes in the Grassmannian under Gauss map. Hold the notations in Section 2 and 3.
%

\textbf{For nonsingular surface $(n=2)$:}

By (\ref{cX}) and  Lemma \ref{cp}, we have
\begin{eqnarray*}
c_1(\gamma^*S)&= & c_1(T_X(-mK_X))+mc_1\\
&= & (2m+1)c_1+mc_1\\
&= & (3m+1)c_1,
\end{eqnarray*}
and
\begin{eqnarray*}
c_2(\gamma^*S)&= & c_2(T_X(-mK_X))+mc_1c_1(T_X(-mK_X))\\
&= & (m^2+m)c_1^2+c_2+mc_1(2m+1)c_1\\
&= & (3m^2+2m)c_1^2+c_2.
\end{eqnarray*}

From Proposition \ref{cS}, it follows
\begin{equation*}
-c_1S=\sigma_1,
\end{equation*}
\begin{equation*}
c_2S=\sigma_{1,1}=\left|
                   \begin{array}{cc}
                     \sigma_1 & \sigma_2 \\
                     1 & \sigma_1 \\
                   \end{array}
                 \right|
                 =\sigma_1^2-\sigma_2,
\end{equation*}

So
\begin{equation*}
0\le c_2S=\sigma_{1,1}=\sigma_1^2-\sigma_2\le \sigma_1^2=(c_1S)^2.
\end{equation*}

Thank for Theorem \ref{Zak}, one sees
\begin{equation*}
0\le c_2(\gamma^*S)\le (c_1(\gamma^*S))^2,
\end{equation*}

i.e.
\begin{equation*}
0\le (3m^2+2m)c_1^2+c_2\le (3m+1)^2c_1^2,
\end{equation*}

i.e.
\begin{equation*}
-(3m^2+2m)c_1^2\le c_2\le (6m^2+4m+1)c_1^2.
\end{equation*}

\begin{remark}
From Bombieri's (\cite{Bom}) or Reider's (\cite{Rei}) result, we know that if characteristic of the field $\kappa$ is $0$ and $K_X$ is ample, then $5K_X$ is very ample.
If characteristic of the field $\kappa$ is positive, by Ekedahl's result (\cite{Eke} or cf. \cite{Ca-Fr}), then we also have $5K_X$ is very ample. So the result is not new. However, if $K_X$ is very ample, we have a uniform formula
\[-5c_1^2\le c_2\le 11c_1^2\]
independent of the characteristic of $\kappa$.
\end{remark}

\textbf{For nonsingular $3$-fold $(n=3)$:}

By (\ref{cX}) and  Lemma \ref{cp}, we have
\begin{eqnarray*}
c_1(\gamma^*S)&= & c_1(T_X(-mK_X))+mc_1\\
&= & (3m+1)c_1+mc_1\\
&= & (4m+1)c_1,
\end{eqnarray*}

\begin{eqnarray*}
c_2(\gamma^*S)&= & c_2(T_X(-mK_X))+mc_1c_1(T_X(-mK_X))\\
&= & (3m^2+2m)c_1^2+c_2+mc_1(3m+1)c_1\\
&= & (6m^2+3m)c_1^2+c_2
\end{eqnarray*}
and
\begin{eqnarray*}
c_3(\gamma^*S)&= & c_3(T_X(-mK_X))+mc_1c_2(T_X(-mK_X))\\
&= & (4m^3+3m^2)c_1^3+2mc_1c_2+c_3\\
&= & m^2(4m+3)c_1^3+2mc_1c_2+c_3.
\end{eqnarray*}

From Proposition \ref{cS}, one obtains
\begin{equation*}
-c_1S=\sigma_1,
\end{equation*}
\begin{equation*}
c_2S=\sigma_{1,1}=\left|
                   \begin{array}{cc}
                     \sigma_1 & \sigma_2 \\
                     1 & \sigma_1 \\
                   \end{array}
                 \right|
                 =\sigma_1^2-\sigma_2
\end{equation*}
and
\begin{equation*}
-c_3S=\sigma_{1,1,1}=\left|
                    \begin{array}{ccc}
                      \sigma_1& \sigma_2 & \sigma_3 \\
                      1 & \sigma_1 & \sigma_2 \\
                      0 & 1 & \sigma_1 \\
                    \end{array}
                  \right|
                  =\sigma_1^3+\sigma_3-2\sigma_1\sigma_2,
\end{equation*}

So
\begin{equation*}
-c_1Sc_2S=\sigma_1\sigma_{1,1}=\sigma_1^3-\sigma_1\sigma_2\le \sigma_1^3=-(c_1S)^3.
\end{equation*}
By Zak's theorem, we can get an inequality in terms of $c_1c_2$ and $c_1^3$:
\begin{equation*}
c_1(\gamma^*S)c_2(\gamma^*S)\ge (c_1(\gamma^*S))^3,
\end{equation*}
i.e.,
\begin{equation*}
(4m+1)c_1((6m^2+3m)c_1^2+c_2)\ge (4m+1)^3c_1^3,
\end{equation*}
i.e.
\begin{equation*}
(4m+1)c_1c_2\ge (4m+1)(10m^2+5m+1)c_1^3.
\end{equation*}

Moreover,
\begin{equation*}
-c_1Sc_2S=\sigma_1\sigma_{1,1}=\sigma_{2,1}+\sigma_{1,1,1}\ge \sigma_{1,1,1}=-c_3S\ge 0.
\end{equation*}
So
\begin{equation*}
c_1(\gamma^*S)c_2(\gamma^*S)\le c_3(\gamma^*S)\le 0,
\end{equation*}
i.e.,
\begin{equation*}
(4m+1)c_1((6m^2+3m)c_1^2+c_2)\le m^2(4m+3)c_1^3+2mc_1c_2+c_3\le 0,
\end{equation*}
i.e.
\begin{equation*}
m(20m^2+15m+3)c_1^3+(2m+1)c_1c_2\le c_3\le -(m^2(4m+3)c_1^3+2mc_1c_2).
\end{equation*}
\begin{remark}
The right hand side of the above inequality is deduced by Hunt in
Section 1 of \cite{Hun} by using Gauss-Bonnet Theorem I (see
\cite{G-H}, Chapter 3.3). But the left hand side of the above
inequality is new.
\end{remark}
\textbf{For nonsingular $4$-folds:}

By (\ref{cX}) and  Lemma \ref{cp}, we obtain
\begin{eqnarray*}
c(T_X(-mK_X))&=&1+(4m+1)c_1+((6m^2+3m)c_1^2+c_2)\\
&&+((4m^3+3m^2)c_1^3+2mc_1c_2+c_3)\\
&&+((m^4+m^3)c_1^4+m^2c_1^2c_2+mc_1c_3+c_4).
\end{eqnarray*}
So
\begin{eqnarray*}
c_1(\gamma^*S)&= & c_1(T_X(-mK_X))+mc_1\\
&= & (4m+1)c_1+mc_1\\
&= & (5m+1)c_1,
\end{eqnarray*}

\begin{eqnarray*}
c_2(\gamma^*S)&= & c_2(T_X(-mK_X))+mc_1c_1(T_X(-mK_X))\\
&= & (6m^2+3m)c_1^2+c_2+mc_1(4m+1)c_1\\
&= & (10m^2+4m)c_1^2+c_2,
\end{eqnarray*}

\begin{eqnarray*}
c_3(\gamma^*S)&= & c_3(T_X(-mK_X))+mc_1c_2(T_X(-mK_X))\\
&= & (4m^3+3m^2)c_1^3+2mc_1c_2+c_3\\
&&+mc_1((6m^2+3m)c_1^2+c_2)\\
&= & (10m^3+6m^2)c_1^3+3mc_1c_2+c_3,
\end{eqnarray*}
and
\begin{eqnarray*}
c_4(\gamma^*S)&= & c_4(T_X(-mK_X))+mc_1c_3(T_X(-mK_X))\\
&= & (m^4+m^3)c_1^4+m^2c_1^2c_2+mc_1c_3+c_4\\
&&+mc_1((4m^3+3m^2)c_1^3+2mc_1c_2+c_3)\\
&= & (5m^4+4m^3)c_1^4+3m^2c_1^2c_2+2mc_1c_3+c_4.
\end{eqnarray*}
From Proposition \ref{cS}, it follows
\begin{equation*}
-c_1S=\sigma_1,
\end{equation*}
\begin{equation*}
c_2S=\sigma_{1,1}=\left|
                   \begin{array}{cc}
                     \sigma_1 & \sigma_2 \\
                     1 & \sigma_1 \\
                   \end{array}
                 \right|
                 =\sigma_1^2-\sigma_2,
\end{equation*}
\begin{equation*}
-c_3S=\sigma_{1,1,1}=\left|
                    \begin{array}{ccc}
                      \sigma_1& \sigma_2 & \sigma_3 \\
                      1 & \sigma_1 & \sigma_2 \\
                      0 & 1 & \sigma_1 \\
                    \end{array}
                  \right|
                  =\sigma_1^3+\sigma_3-2\sigma_1\sigma_2,
\end{equation*}
and
\begin{equation*}
c_4S=\sigma_{1,1,1,1}=\left|
                     \begin{array}{cccc}
                       \sigma_1 & \sigma_2  & \sigma_3& \sigma_4 \\
                       1 & \sigma_1 & \sigma_2  & \sigma_3 \\
                       0 & 1 & \sigma_1 & \sigma_2  \\
                       0 & 0 & 1 & \sigma_1 \\
                     \end{array}
                   \right|
                  =\sigma_1^4-3\sigma_1^2\sigma_2+2\sigma_1\sigma_3+\sigma_2^2-\sigma_4.
\end{equation*}
Similarly, we can have several inequalities.

(1) We can get an inequality in terms of $c_1^2c_2$ and $c_1^4$:
\begin{equation*}
0\le (c_1S)^2c_2S=\sigma_1^2\sigma_{1,1}=\sigma_1^4-\sigma_1^2\sigma_2\le \sigma_1^4=(c_1S)^4,
\end{equation*}
so
\begin{equation*}
0\le (c_1(\gamma^*S))^2c_2(\gamma^*S)\le (c_1(\gamma^*S))^4,
\end{equation*}
i.e.,
\begin{equation*}
0\le (5m+1)^2c_1^2((10m^2+4m)c_1^2+c_2)\le (5m+1)^4c_1^4,
\end{equation*}
i.e.,
\begin{equation*}
-(5m+2)c_1^4\le c_1^2c_2\le (15m^2+6m+1)c_1^4.
\end{equation*}

(2) We can also get an inequality in terms of $c_1^2c_2$, $c_1c_3$ and $c_1^4$:
\begin{eqnarray*}
0\le c_1Sc_3S&=&\sigma_1\sigma_{1,1,1}\\
             &=&\sigma_1^4+\sigma_1\sigma_3-2\sigma_1^2\sigma_2\\
             &=&\sigma_1^4+\sigma_1\sigma_3-2\sigma_1(\sigma_3+\sigma_{2,1})\\
             &=&\sigma_1^4-\sigma_1\sigma_3-2\sigma_1\sigma_{2,1}\le \sigma_1^4=(c_1S)^4,
\end{eqnarray*}

so
\begin{equation*}
0\le c_1(\gamma^*S)c_3(\gamma^*S)\le (c_1(\gamma^*S))^4,
\end{equation*}
i.e.
\begin{equation*}
0\le (5m+1)c_1((10m^3+6m^2)c_1^3+3mc_1c_2+c_3)\le (5m+1)^4c_1^4,
\end{equation*}
i.e.
\begin{equation*}
-2m^2(5m+3)c_1^4-3mc_1^2c_2\le c_1c_3\le (115m^3+69m^2+15m+1)c_1^4-3mc_1^2c_2.
\end{equation*}
(3) We can also get an inequality in terms of $c_1^2c_2$ , $c_2^2$ and $c_1^4$:
\begin{eqnarray*}
0\le (c_2S)^2&=&\sigma_{1,1}^2\\
             &=&(\sigma_1^2-\sigma_2)^2\\
             &=&\sigma_1^4-2\sigma_1^2\sigma_2+\sigma_2^2\\
             &=&\sigma_1^4-2(\sigma_{1,1}+\sigma_2)\sigma_2+\sigma_2^2\\
             &=&\sigma_1^4-2\sigma_{1,1}-\sigma_2^2\le \sigma_1^4=(c_1S)^4,
\end{eqnarray*}

so
\begin{equation*}
0\le (c_2(\gamma^*S))^2\le (c_1(\gamma^*S))^4,
\end{equation*}
i.e.
\begin{equation*}
0\le ((10m^2+4m)c_1^2+c_2)^2\le (5m+1)^4c_1^4,
\end{equation*}
i.e.
\begin{eqnarray*}
&&-4m^2(5m+2)^2c_1^4-4m(5m+2)c_1^2c_2\le c_2^2\\
&\le& ((5m+1)^4-4m^2(5m+2)^2)c_1^4-4m(5m+2)c_1^2c_2.
\end{eqnarray*}
(4) We can also get an inequality in terms of $c_1^2c_2$ , $c_1c_3$, $c_4$ and $c_1^4$:
\begin{eqnarray*}
0\le c_4S    &=&\sigma_{1,1,1,1}\\
             &=&\sigma_1^4-3\sigma_1^2\sigma_2+2\sigma_1\sigma_3+\sigma_2^2-\sigma_4\\
             &=&\sigma_1^4-2(\sigma_1\sigma_3+\sigma_{2,1}\sigma_1)-(\sigma_2^2+\sigma_{1,1}\sigma_2)\\
             &&+2\sigma_1\sigma_3+\sigma_2^2-\sigma_4\\
             &=&\sigma_1^4-2\sigma_{2,1}\sigma_1-\sigma_{1,1}\sigma_2-\sigma_4\le \sigma_1^4=(c_1S)^4,
\end{eqnarray*}

so
\begin{equation*}
0\le c_4(\gamma^*S)\le (c_1(\gamma^*S))^4,
\end{equation*}
i.e.,
\begin{equation*}
0\le(5m^4+4m^3)c_1^4+3m^2c_1^2c_2+2mc_1c_3+c_4\le (5m+1)^4c_1^4,
\end{equation*}
i.e.,
\begin{equation*}
-(5m^4+4m^3)c_1^4\le m^2c_1^2c_2+2mc_1c_3+c_4\le ((5m+1)^4-(5m^4+4m^3))c_1^4.
\end{equation*}
(5) We can also get an inequality in terms of $c_1^2c_2$ , $c_2^2$ and $c_1^4$:
\begin{eqnarray*}
(c_1S)^2c_2S &=&\sigma_1^2\sigma_{1,1}\\
             &=&\sigma_1^4-\sigma_1^2\sigma_2\\
             &=&\sigma_1^4-2\sigma_1^2\sigma_2+(\sigma_2+\sigma_{1,1})\sigma_2\\
             &\ge&\sigma_1^4-2\sigma_1^2\sigma_2+\sigma_2^2\\
             &=&(\sigma_1^2-\sigma_2)^2=(c_2S)^2,
\end{eqnarray*}
so
\begin{equation*}
(c_2(\gamma^*S))^2\le (c_1(\gamma^*S))^2c_2(\gamma^*S),
\end{equation*}
i.e.
\begin{equation*}
((10m^2+4m)c_1^2+c_2)^2 \le (5m+1)^2c_1^2((10m^2+4m)c_1^2+c_2),
\end{equation*}

i.e.

\begin{equation*}
c_2^2 \le 2m(5m+2)(15m^2+6m+1)c_1^4+(5m^2+2m+1)c_1^2c_2.
\end{equation*}
(6) We can also get an inequality in terms of $c_1^2c_2$ , $c_2^2$, $c_1c_3$ and $c_1^4$:
\begin{eqnarray*}
(c_2S)^2-c_1Sc_3S&=&(\sigma_1^2-\sigma_2)^2-\sigma_1(\sigma_1^3+\sigma_3-2\sigma_1\sigma_2)\\
                 &=& \sigma_2^2-\sigma_1\sigma_3\\
                 &=&(\sigma_4+\sigma_{3,1}+\sigma_{2,2})-(\sigma_4+\sigma_{3,1})\\
                 &=&\sigma_{2,2}\ge 0,
\end{eqnarray*}
so
\begin{equation*}
c_1(\gamma^*S)c_3(\gamma^*S)\le (c_2(\gamma^*S))^2.
\end{equation*}
i.e.
\begin{equation*}
(5m+1)c_1((10m^3+6m^2)c_1^3+3mc_1c_2+c_3)\le ((10m^2+4m)c_1^2+c_2)^2,
\end{equation*}

i.e.

\begin{equation*}
c_1c_3 \le 10m^2(5m^2+4m+1)c_1^4+5m(4m+1)c_1^2c_2+c_2^2.
\end{equation*}
(7) We can also get an inequality in terms of $c_1^2c_2$ , $c_4$, $c_1c_3$ and $c_1^4$:
\begin{eqnarray*}
c_1Sc_3S-c_4S    &=&\sigma_1(\sigma_1^3+\sigma_3-2\sigma_1\sigma_2)\\
&&-(\sigma_1^4-3\sigma_1^2\sigma_2+2\sigma_1\sigma_3+\sigma_2^2-\sigma_4)\\
                 &=& \sigma_1^2\sigma_2-\sigma_1\sigma_3-\sigma_2^2+\sigma_4\\
                 &=&\sigma_1\sigma_{2,1}+\sigma_4-(\sigma_4+\sigma_{3,1}+\sigma_{2,2})\\
                 &=&\sigma_{2,1,1}\ge 0,
\end{eqnarray*}

so
\begin{equation*}
c_4(\gamma^*S)\le c_1(\gamma^*S)c_3(\gamma^*S),
\end{equation*}
i.e.,
\begin{equation*}
(5m^4+4m^3)c_1^4+3m^2c_1^2c_2+2mc_1c_3+c_4\le(5m+1)c_1((10m^3+6m^2)c_1^3+3mc_1c_2+c_3),
\end{equation*}

i.e.

\begin{equation*}
c_4\le 3m^2(15m^2+12m+2)c_1^4+3m(4m+1)c_1^2c_2+(3m+1)c_1c_3.
\end{equation*}

\textbf{For general dimension $n$:}

Let $c_0S=\sigma_0=1$, by Lemma \ref{det}, Giambelli's
formula and Proposition \ref{cS}, one obtains

$$c_nS\sigma_0+c_{n-1}S\sigma_1+c_{n-2}S\sigma_2+\cdots+c_1S\sigma_{n-1}+c_0S\sigma_n=0,$$ for any positive integer $n$.
This implies $$\sum_{i=0}^\infty c_iSx^i\sum_{i=0}^\infty \sigma_ix^i=1+\sum_{j=1}^\infty(\sum_{i=0}^jc_iS\sigma_{j-i})x^j=1,$$ where $x$ is an indeterminate. It follows that
\begin{eqnarray*}
1+\sum_{i=1}^\infty \sigma_ix^i &=& \frac{1}{1+\sum_{i=1}^\infty c_iSx^i}\\
&=&1+\sum_{k=1}^\infty(-1)^k\left(\sum_{i=1}^\infty c_iSx^i\right)^k\\
&=&1+\sum_{k=1}^\infty(-1)^k\left(\sum_{m=1}^\infty\sum_{\scriptstyle i_1+  \cdots+i_k=m \atop \scriptstyle i_1,\cdots,i_k\geq 1} c_{i_1,\cdots ,i_k}Sx^m \right)\\
&=&1+\sum_{m=1}^\infty \sum_{k=1}^m\left((-1)^k \sum_{\scriptstyle i_1+  \cdots+i_k=m \atop \scriptstyle i_1,\cdots,i_k\geq 1} c_{i_1,\cdots ,i_k}S\right)x^m
\end{eqnarray*}
Therefore, one sees
\begin{eqnarray}\label{sigmaToc}
\sigma_m &=& \sum_{k=1}^m\left((-1)^k \sum_{\scriptstyle
i_1+\cdots+i_k=m \atop \scriptstyle i_1,\cdots,i_k\geq 1}
c_{i_1,\cdots,i_k}S\right)\\\nonumber
&=& \sum_{\scriptstyle j_1+2j_2+\cdots+mj_m=m \atop \scriptstyle
j_1,\cdots,j_m\geq
0}(-1)^{j_1+\cdots+j_m}\frac{(j_1+\cdots+j_m)!}{j_1!\cdots
j_m!}c_{1^{j_1}}S\cdots c_{m^{j_m}}S.
\end{eqnarray}

In particular, we have
\begin{eqnarray}
-c_1S=\sigma_1\\
c_1^2S-c_2S=\sigma_2\\
-c_1^3S+2c_1Sc_2S-c_3S=\sigma_3\\
c_1^4S-3c_1^2Sc_2S+2c_1Sc_3S+c_2^2S-c_4S=\sigma_4
\end{eqnarray}
\begin{center}
$\cdots\cdots~~\cdots\cdots$
\end{center}

We can have a bunch of inequalities of Chern classes.

Step 1:  Express $n$ as the summation of positive integers without considering the orders of them, say $a_1\ge a_2\ge\cdots\ge a_r>0$, such that $n\ge \sum_{i=1}^r a_i$. Let $\overrightarrow{a}=(a_1,a_2,\cdots,a_r)$.

Step 2: By Giambelli's Formula, write $\sigma_{\overrightarrow{a}}$ in terms of $\sigma_1, \sigma_2, \cdots, \sigma_n$.

Step 3: By (\ref{sigmaToc}), express $\sigma_{\overrightarrow{a}}$ in terms of $c_1S,\cdots, c_nS$.

Step 4: Since $\sigma_{\overrightarrow{a}}>0$, by Zak's Theorem \ref{Zak}, pull back of the $\sigma_{\overrightarrow{a}}$ under Gauss map $\gamma$, we can express $\gamma^*(\sigma_{\overrightarrow{a}})$ in terms of $\gamma^*(c_1S),\cdots, \gamma^*(c_nS)$ which is great than or equal to $0$.

Stpe 5: Express $\gamma^*(c_1S),\cdots, \gamma^*(c_nS)$ in terms of $c_1, c_2, \cdots, c_n$ by (\ref{cX}) and Lemma \ref{cp} and get a inequality finally.

\begin{example}
Suppose $n=5$.

Step 1: Consider $\overrightarrow{a}=(3,2)$;

Step 2: Calculate \[\sigma_{3,2}=\left|
                       \begin{array}{cc}
                         \sigma_3 & \sigma_4 \\
                        \sigma_1 & \sigma_2 \\
                       \end{array}
                     \right|
                     =\sigma_3\sigma_2-\sigma_1\sigma_4;
\]

Step 3: We have \begin{eqnarray*}
0\le\sigma_{3,2}    &=& (-c_1^3S+2c_1Sc_2S-c_3S)(c_1^2S-c_2S)\\
&&+c_1S(c_1^4S-3c_1^2Sc_2S+2c_1Sc_3S+c_2^2S-c_4S)\\
               &=&c_1^2Sc_3S-c_1Sc_2^2S+c_3Sc_2S-c_1Sc_4S;
\end{eqnarray*}

Step 4: \[(\gamma^*(c_1S))^2\gamma^*(c_3S)-\gamma^*(c_1S)(\gamma^*(c_2S))^2+\gamma^*(c_3S)\gamma^*(c_2S)-\gamma^*(c_1S)\gamma^*(c_4S)\ge 0.\]

Step 5: By (\ref{cX}) and  Lemma \ref{cp},
\begin{eqnarray*}
c_1(\gamma^*S)&= & c_1(T_X(-mK_X))+mc_1\\
&= & (5m+1)c_1+mc_1\\
&= & (6m+1)c_1,
\end{eqnarray*}

\begin{eqnarray*}
c_2(\gamma^*S)&= & c_2(T_X(-mK_X))+mc_1c_1(T_X(-mK_X))\\
&= & (10m^2+4m)c_1^2+c_2+mc_1(5m+1)c_1\\
&= & (15m^2+5m)c_1^2+c_2,
\end{eqnarray*}

\begin{eqnarray*}
c_3(\gamma^*S)&= & c_3(T_X(-mK_X))+mc_1c_2(T_X(-mK_X))\\
&= & (10m^3+6m^2)c_1^3+3mc_1c_2+c_3\\
&&+mc_1((10m^2+4m)c_1^2+c_2)\\
&= & (20m^3+10m^2)c_1^3+4mc_1c_2+c_3,
\end{eqnarray*}
and
\begin{eqnarray*}
c_4(\gamma^*S)&= & c_4(T_X(-mK_X))+mc_1c_3(T_X(-mK_X))\\
&= & (5m^4+4m^3)c_1^4+3m^2c_1^2c_2+2mc_1c_3+c_4\\
&&+mc_1((10m^3+6m^2)c_1^3+3mc_1c_2+c_3)\\
&= & (15m^4+10m^3)c_1^4+6m^2c_1^2c_2+3mc_1c_3+c_4.
\end{eqnarray*}
Finally, we get
\[\begin{split}
-(420m^5&+350m^4+120m^3+15m^2)c_1^5+(8m^3-18m^2-6m)c_1^3c_2\\
&+(33m^2+14m+1)c_1^2c_3-(2m+1)c_1c_2^2\\
&-(6m+1)c_1c_4+c_2c_3\ge 0.
\end{split}\]
\end{example}

If $K_X$ is very ample, i.e. $m=1$ , them we have \[-905c_1^5-16c_1^3c_2+48c_1^2c_3-3c_1c_2^2-7c_1c_4+c_2c_3\ge 0\]
\begin{remark}
For fixed dimension $n$, the number of the inequalities of Chern classes is $\sum_{i=2}^np(n)$, where $p(n)$ is the partition number of $n$. Partitions can be graphically visualized with Young diagrams or Ferrers diagrams. They occur in a number of branches of mathematics and physics, including the study of symmetric polynomials, the symmetric group and in group representation theory in general. It is known that (cf. \cite{And}) an asymptotic expression for $p(n)$ is given by
\[ {\displaystyle p(n)\sim {\frac {1}{4n{\sqrt {3}}}}\exp \left({\pi {\sqrt {\frac {2n}{3}}}}\right)} ~as~ {\displaystyle n\rightarrow \infty }.\]
\end{remark}

\section*{Acknowledgements}
The first author would like to thank  N. Mok for providing excellent research environment in the University of Hong Kong while part of
this research was done there. Both authors would like to thank for the reviewers for pointing out some typos and useful suggestions in the original version.


\begin{thebibliography}{Du-Sun}
\bibitem[And]{And} G. Andrews: \emph{The Theory of Partitions}, Cambridge University Press, 1976.

\bibitem[Bad]{Bad} L. Badescu: \emph{Algebraic surfaces}, Universitext Vol. 207, Springer (2001).

\bibitem[Bom]{Bom} E. Bombieri: \emph{Canonical models of surfaces of general type}, Publ. Math. IHES,
42 (1973), 171–219.

\bibitem[B-S-S]{B-S-S} M. Beltrametti, M. Schneider, A. Sommese: \emph{Chern inequalities and spannedness of adjoint bundles}, Proceedings of the Hirzebruch 65 Conference on Algebraic Geometry (Ramat Gan, 1993), 97-107, Israel Math. Conf. Proc., 9, Bar-Ilan Univ., Ramat Gan, 1996.

\bibitem[Ca-Fr]{Ca-Fr} F. Catanese and M. Franciosi, \emph{Divisors of small genus on algebraic surfaces and
projective embeddings}, Proceedings of the conference “Hirzebruch 65”, Tel Aviv 1993, Contemp. Math., A.M.S. (1994), subseries ‘Israel Mathematical Conference
Proceedings’ Vol. 9 (1996), 109–140.

\bibitem[Ca-Sc]{Ca-Sc} F. Catanese, M. Schneider: \emph{Bounds for stable bundles and degrees of Weiersteass schemes}, Math. Ann. 293 (1992), 579-594.

\bibitem[Ch-Le]{Ch-Le}K. Chan, N. Leung: \emph{Miyaoka-Yau-type inequalities for K$\acute{a}$hler-Einstein manifolds}, Comm. Anal. Geom. 15 (2007), no. 2, 359-379.

\bibitem[CJ-H]{CJ-H} J. Chen, C. Hacon: \emph{On the geography of threefolds of general type}, J. Algebra 321 (2009), no. 9, 2500-2507.

\bibitem[C-K-N]{C-K-N} M.C. Chang, H. Kim, S. Nollet: \emph{Bounds on $c_3$ for threefolds}. Manuscripta Math. 97, 135-141 (1998).

\bibitem[C-L]{C-L} M. Chang, A. Lopez: \emph{A linear bound on the Euler number of threefolds of Calabi-Yau and of general type}, Manuscripta Math. 105 (2001), no. 1, 47-67.

\bibitem[CM-H]{CM-H} M. Chen, C. Hacon: \emph{On the geography of Gorenstein minimal 3-folds of general type}, Asian J. Math. 10 (2006), no. 4, 757-763.

\bibitem[Dem1]{Dem1} J.-P. Demailly: \emph{A numerical criterion for very ample line bundles}, J. Differential Geom., 37 (1993), 323-374.

\bibitem[Dem2]{Dem2} J.-P. Demailly: \emph{ Effective bounds for very ample line bundles}, Invent. Math. 124 (1996), 243-261.

\bibitem[D-P-S]{D-P-S} J.-P. Demailly, T.Peternell, M. Schneider: \emph{Compact complex manifolds with numerically effective tangent bundles}, J. Algebraic Geom. 3 (1994), no. 2, 295-345.

\bibitem[Eas]{Eas} R. Easton: \emph{Surfaces violating Bogomolov-Miyaoka-Yau in positive characteristic}, Proceedings of the American Mathematical Society, (7) 136 (2008), 2271-2278.

\bibitem[E-H]{E-H} D. Eisenbud and J. Harris: \emph{3264 and All That: A Second Course in Algebraic Geometry}, Cambridge University Press, Cambridge, 2016. xiv+616 pp.

\bibitem[Eke]{Eke} T. Ekedahl: \emph{Canonical models of surfaces of general type in positive characteristic},
Publ. Math. IHES, 67 (1988), 97–144.

\bibitem[E-L-N]{E-L-N} L. Ein, R. Lazarsfeld, M. Nakamaye: \emph{Zero-estimates, intersection theory, and a theorem of Demailly}. Higher-dimensional complex varieties (Trento, 1994), 183-207, de Gruyter, Berlin, 1996.

\bibitem[F-L]{F-L} W. Fulton and R. Lazarsfeld: \emph{Positive polynomials for ample vector bundles}, Ann. of Math. (2) 118 (1983), no. 1, 35–60.

\bibitem[Fuj]{Fuj} T. Fujita: \emph{On polarized manifolds whose adjoint bundles are not semipositive}, Algebraic Geometry, Sendai, 1985, Adv. Stud. in Pure Math., North Holland, T. Oda (ed.), 10 (1987), 167-178.

\bibitem[G-H]{G-H} P. Griffiths, J. Harris: \emph{Principles of algebraic geometry}, Wiley, New York, 1978.

\bibitem[Gu]{Gu}  Y. Gu: \emph{On algebraic surfaces of general type with negative $c_2$}, Compos. Math. 152 (2016), no. 9, 1966-1998.

\bibitem[Hun]{Hun} B. Hunt: \emph{Complex manifold geography in dimension 2 and 3}, J. Differential Geom. 30 (1989), no. 1, 51-153.

\bibitem[Kee]{Kee} D. Keeler, \emph{Fujita's conjecture and Frobenius amplitude}, Amer. J. Math. 130 (2008), no. 5, 1327-1336.

\bibitem[Ko]{Ko} D. Kotschick: \emph{Chern numbers and diffeomorphism types of projective varieties}, J. Topol. 1 (2008), no. 2, 518-526.

\bibitem[Lan]{Lan} W. Lang: \emph{Examples of surfaces of general type with vector fields}, Arithmetic and geometry, Vol. II, Progr. Math., 36, Boston, MA: Birkh$\ddot{a}$user Boston, pp. (1983), 167-173.

\bibitem[Lie]{Lie} C. Liedtke: \emph{Algebraic surfaces of general type with small $c_2$ in positive characteristic}, Nagoya Math.J, 191 (2008), 111-134.

\bibitem[Mi]{Mi} Y. Miyaoka: \emph{Themes and variations��on inequalities of Chern classes}. (Japanese) S$\bar{u}$gaku 41 (1989), no. 3, 193-207.

\bibitem[Ray]{Ray} M. Raynaud: \emph{Contre-exemple au vanishing theorem en caract$\acute{e}$ristique $p>0$}, Tata Inst.
Fund. Res. Studies in Math., 8, Berlin, New York: Springer-Verlag (1978), 273-278.

\bibitem[Rei]{Rei} I. Reider: \emph{Vector bundles of rank 2 and linear systems on algebraic surfaces}, Ann. of Math., 127 (1988), 309-316.

\bibitem[Per]{Per} U. Persson: \emph{Chern invariants of surfaces of general type}, Compositio Math. 43 (1981) 3-58.

\bibitem[S-B]{S-B} N. I. Shepherd-Barron: \emph{Geography for surfaces of general type in positive characteristic},
Invent. Math.,106(1) (1991),263-274.

\bibitem[Siu1]{Siu1} Y.-T. Siu: \emph{An effective Matsusaka big theorem}, Ann. Inst. Fourier (Grenoble) 43 (1993), no. 5, 1387-1405.

\bibitem[Siu2]{Siu2} Y.-T. Siu: \emph{Very ampleness criterion of double adjoints of ample line bundles}, Modern methods in complex analysis (Princeton, NJ, 1992), 291-318, Ann. of Math. Stud., 137, Princeton Univ. Press, Princeton, NJ, 1995.

\bibitem[Siu3]{Siu3} Y.-T. Siu: \emph{Effective very ampleness}, Invent. Math. 124 (1996), 563-571.

\bibitem[Smi1]{Smi1} K. Smith: \emph{Fujita's freeness conjecture in terms of local cohomology}, J. Algebraic Geom. 6 (1997), no. 3, 417-429.

\bibitem[Smi2]{Smi2} K. Smith: \emph{A tight closure proof of Fujita's freeness conjecture for very ample line bundles},
Math. Ann. 317 (2000), no. 2, 285-293.

\bibitem[Sun]{Sun} H. Sun: \emph{Tilt-stability, vanishing theorems and Bogomolov-Gieseker type inequalities}, arXiv:1609.03245.

\bibitem[Szp]{Szp} L. Szpiro: \emph{Sur le th$\acute{e}$or$\grave{e}$me de rigidit$\acute{e}$ de Parsin et Arakelov}.

\bibitem[Yeu]{Yeu} S.-K. Yeung: \emph{Very ampleness of line bundles and canonical embedding of coverings of manifolds},
Compositio Math. 123 (2000), no. 2, 209–223.

\bibitem[Zak]{Zak}F. L. Zak: \emph{Tangents and secants of algebraic varieties},(English summary)
Translated from the Russian manuscript by the author. Translations of Mathematical Monographs, 127. American Mathematical Society, Providence, RI, 1993. viii+164 pp.
\end{thebibliography}
\end{document}